\documentclass[12pt]{amsart}
\usepackage{amsfonts, amssymb, amscd}

\def\dual                 {{\vee}}

\def\ee                 {{\rm e}}

\def\RR                 {{\mathbb R}}
\def\CC                 {{\mathbb C}}

\newtheorem{lemma}{Lemma}[section]
\newtheorem{theorem}[lemma]{Theorem}
\newtheorem{corollary}[lemma]{Corollary}
\newtheorem{proposition}[lemma]{Proposition}
\theoremstyle{definition}
\newtheorem{definition}[lemma]{Definition}

\newtheorem{remark}[lemma]{Remark}
\theoremstyle{remark}
\newtheorem*{proof*}{Proof}
\numberwithin{equation}{section}

\title{On stringy cohomology spaces}
\author{Lev A. Borisov}
\address{Mathematics Department, Rutgers University, 110 Frelinghuysen Rd, Piscataway, NJ 08540, USA}
\email{borisov@math.rutgers.edu}

\begin{document}
\begin{abstract}
We modify the definition of the families of $A$ and $B$ 
stringy cohomology spaces associated to a pair of
dual reflexive Gorenstein cones. The new spaces have the same dimension as the ones defined in the joint paper with Mavlyutov \cite{BM}, but they admit natural flat connections
with respect to the appropriate parameters. This solves a longstanding question of relating GKZ hypergeometric system to stringy cohomology. 
We construct products on these spaces
by vertex algebra techniques. In the process, we fix a minor gap in \cite{BM} and
prove a statement on intersection cohomology of dual cones that may be of independent interest.
\end{abstract}

\maketitle

\section{Introduction}\label{intro}
In the early history of mirror symmetry, the duality of Hodge numbers played an important role.
Specifically, for three-dimensional mirror Calabi-Yau varieties $X$ and $X^\dual$ one 
has $h^{1,1}(X)=h^{1,2}(X^\dual)$ and $h^{1,2}(X)=h^{1,1}(X^\dual)$. In general, for $n$-dimensional
mirror Calabi-Yau varieties one may expect
\begin{equation}\label{1}
h^{p,q}(X)=h^{n-p,q}(X^\dual),
\end{equation}
however there are complications. The initial obstacle is that higher dimensional singular  Calabi-Yau 
varieties often do not admit crepant resolutions, and if one tries to verify  \eqref{1}
for singular mirrors, it simply fails. In this sense, usual Hodge numbers are ill-suited for higher dimensional mirror symmetry.

An elegant solution to this difficulty has been obtained in \cite{Batyrev-Dais, Batyrev-stringy}.
 Namely, for any $X$ with log-terminal singularities
one introduces the stringy $E$-function $E(u,v)$ as a certain weighted sum over strata of a log resolution of $X$ and shows that it is independent of the resolution. In the case when $X$ is toroidal, the stringy $E$-function turns out to be a polynomial, which allows one to define stringy Hodge numbers 
$h^{p,q}_{st}(X)$ as coefficients of $E$. These numbers satisfy several nice properties.
If $X$ admits a crepant resolution $Y$, then $h_{st}^{p,q}(X)
=h^{p,q}(Y)$. If $X$ admits a crepant resolution by an orbifold (smooth DM stack) then $h^{p,q}(X)$ 
equals the dimension of the corresponding orbifold cohomology of $Y$. 
Perhaps the best justification for working with $h^{p,q}_{st}$ comes from the large class of mirror symmetry examples provided by hypersurfaces and complete intersections in Gorenstein toric Fano varieties. For these examples, the duality 
\begin{equation}\label{2}
h^{p,q}_{st}(X)=h^{n-p,q}_{st}(X^\dual)
\end{equation}
has been proved in \cite{BB}.

Once the appropriate stringy Hodge numbers are defined,
the next natural question to ask is what are the stringy cohomology \emph{vector spaces} whose
dimensions are given by  $h^{p,q}_{st}$. 
This is already a meaningful question in dimension three. While one can consider cohomology of 
a crepant resolution to be a sensible answer, the non-uniqueness of such resolution presents 
a problem. One expects to have not one stringy cohomology space but a family of such spaces
that interpolates between cohomology of different crepant (orbifold) resolutions, if any exist. While in general  the 
problem of constructing stringy cohomology
 is wide open, reasonable definitions can be made in the setting of toric mirror symmetry.
Specifically, Definition \ref{firstBM} below was proposed in \cite{BM}. 
  To explain the ingredients of this definition we first recall the combinatorial data
  of toric mirror symmetry.

Let $M$ and $N$ be dual free abelian groups (often referred to as lattices). Let $K$ and $K^\dual$
 be dual maximum-dimensional rational polyhedral cones in $M$ and $N$. These dual cones are called
 reflexive Gorenstein if the lattice 
 generators of the one-dimensional faces of both $K$ and $K^\dual$ lie in a
 hyperplane  at distance one from the origin in $M$ and $N$ respectively. This condition on the cones $K$ and $K^\dual$ is equivalent to the statement that both semigroup rings $\CC[K]$ and $\CC[K^\dual]$ are Gorenstein. 
 
Dual reflexive Gorenstein cones $K$ and $K^\dual$ uniquely determine a pair of lattice points
$\deg\in M$ and $\deg^\dual\in N$ which define the aforementioned hyperplanes by 
$\{\bullet\cdot \deg^\dual =1\}$ and $\{\deg\cdot\bullet=1\}$. 
We denote by $\Delta$ (respectively $\Delta^\dual)$ the set of lattice points in $M$ (respectively $N$)
that lie in these hyperplanes. 
For generic choices of coefficient functions $f:\Delta\to \CC$ and $g:\Delta^\dual\to \CC$,  one can
typically\footnote{There 
are some obstructions to this in $\deg\cdot\deg^\dual>1$ case, see \cite{BatNill}. However,
stringy cohomology vector spaces can still be constructed in the absence of Calabi-Yau varieties.
This suggests that the underlying $N=(2,2)$ superconformal field theories can also be 
constructed.} define mirror families of singular Calabi-Yau varieties $(X_f)$ and $(X_g^\dual)$. 

Definition \ref{firstBM} of stringy cohomology vector spaces in the setting of dual Gorenstein cones 
has been first suggested by Mavlyutov. Their dimensions have been consequently calculated in \cite{BM}
to coincide with the stringy Hodge numbers of $X_f$ and $X_g^\dual$ whenever they exist.
Here is the description of the construction. For a face $\theta$ of $K$ we consider the ideal $I_{f,\theta}$ in the semigroup ring
$\CC[\theta]$ generated by
the elements $\sum_{m\in\theta} f(m)\mu(m)[m]$ for all linear functions $\mu$ on the span of $\theta$. 
Here the summation is taken over elements in $\Delta$ that lie in $\theta$.
Then the  
space $R_1(f,\theta)$ is defined as the image of 
$$
\CC[\theta^{\circ}]/ I_{f,\theta} \CC[\theta^{\circ}] \to \CC[\theta]/I_{f,\theta} \CC[\theta]
$$
(see \cite{BM}) where $\theta^\circ$ is the interior of $\theta$ and thus $\CC[\theta^\circ]$
is an ideal in $\CC[\theta]$. To a face $\theta$ of $K$ we associate a face $\theta^* :={\rm Ann}(\theta)\cap K^\dual$ of
$K^\dual$ (we use $\theta^*$ to avoid the confusion with the dual cone $\theta^\dual$) and  define $R_1(g,\theta^*)$ similarly.

\begin{definition}\label{firstBM}
Stringy cohomology space of $(X_f,X_g^\dual)$ is given by
$$
\bigoplus_{ \{0\}\subseteq \theta \subseteq K} R_1(f,\theta)\otimes R_1(g,\theta^*) \otimes \Lambda^{\dim \theta^*} (\CC\theta^*).
$$
\end{definition}

\medskip
The space above possesses a natural double grading, see \cite{BM}\footnote{For exposition reasons, 
only the case of $\deg\cdot\deg^\dual=1$ case was considered in \cite{BM} but the arguments apply to any pair
of reflexive Gorenstein cones.}
such that the stringy Hodge numbers can be recovered as dimensions of double-graded components.  While it is satisfying to have natural spaces with the appropriate double
grading, Definition \ref{firstBM} raises two important questions.

\begin{itemize}
\item One expects to have two different  graded super-commutative associative products
on the stringy cohomology vector spaces. How can one construct these products?

\item How to  relate these spaces to the GKZ hypergeometric system of partial differential equations?

\end{itemize}

The first issue has been partially resolved in \cite{BM} and \cite{chiralrings} by identifying 
the spaces of Definition \ref{firstBM} with the chiral rings of certain $N=2$ vertex algebras under an additional technical assumption 
that $f$ and $g$ are  \emph{strongly} nondegenerate, which is a somewhat smaller open set than 
that of nondegenerate coefficient functions.  This interpretation of the stringy cohomology vector space as a chiral ring hinges on the following construction of this 
space  as cohomology of a natural complex of vector spaces.
Consider the quotient
$\CC[(K\oplus K^\dual)_0]$ of the semigroup ring $\CC[K\oplus K^\dual]$ by the ideal spanned by  $[m,n],
m\cdot n >0$. Define 
$$V=\CC[(K\oplus K^\dual)_0]\otimes \Lambda^* N_\CC$$
and an endomorphism $d_{f,g}:V\to V$ by 
$$
d_{f,g}=\sum_{m\in \Delta} f(m) [m] \otimes ({\rm contr.} m) + \sum_{n\in\Delta^\dual} g(n)[n]\otimes (n\wedge).
$$
It is easy to see that $d_{f,g}^2=0$. Moreover, $d$ increases by  one the natural grading on $V$ given by 
$[m,n]\otimes P\to m\cdot \deg^\dual +\deg\cdot\, n$.  The following result was first claimed in \cite{BM} and then used in \cite{chiralrings} to identify the space of Definition \ref{firstBM} with
the chiral ring of a vertex algebra.

\medskip\noindent
{\bf Theorem \ref{main}.}
The cohomology of $V$ with respect to $d_{f,g}$ is naturally isomorphic to the direct sum over faces $\theta$ of $K$
$$
\bigoplus_{ \{0\}\subseteq \theta \subseteq K} R_1(f,\theta)\otimes R_1(g,\theta^*) \otimes \Lambda^{\dim \theta^*} (\CC\theta^*).
$$
\smallskip

\noindent
Unfortunately, the proof of Theorem \ref{main}
presented in \cite{BM} was incorrect, and the correction is given in this paper
as a consequence of a result in intersection cohomology 
of polyhedral fans. 

The second issue with Definition \ref{firstBM}, namely the relation to the GKZ hypergeometric system
has eluded understanding for too long now. Finally, we are able to resolve it satisfactorily by modifying 
the complex of \cite{BM} slightly. Specifically, the main results of this paper are the following 
Definition 
\ref{defmain}, Theorem \ref{mainGKZ} and Corollary \ref{maincoro5}.

\medskip\noindent
{\bf Definition \ref{defmain}.}
Let $K$ and $K^\dual$ be dual reflexive Gorenstein cones. Define as before the space 
$V=\CC[(K\oplus K^\dual)_0]\otimes \Lambda^* N_\CC$. Consider the 
differential $\widehat d_{f,g}$ on it given by 
$$
\widehat d_{f,g} ([m_1\oplus n_1]\otimes P) = d_{f,g}([m_1\oplus n_1]\otimes P)+ [m_1\oplus n_1]\otimes (n_1\wedge P)
$$
$$
=\sum_{m\in \Delta} f(m) [(m+m_1)\oplus n_1] \otimes ({\rm contr.} m)(P) \hskip 100pt
$$$$+ \sum_{n\in\Delta^\dual} g(n)
[m_1\oplus (n+n_1)]\otimes (n\wedge P)
+[m_1\oplus n_1]\otimes (n_1\wedge P)
.
$$
Then we define stringy cohomology $B$-space $H_{B}(X_f,X_g^\dual)$ as the cohomology
of $V$ with respect to $\widehat d_{f,g}$, with the grading given by 
$$
[m\oplus n]\otimes P \mapsto 2 \,m\cdot \deg^\dual + \deg(P).
$$
We will also call this the $A$-space of the mirror pair $(X_g^\dual,X_f)$. Similarly, 
 $H_A(X_f,X_g^\dual)$ (equal to  $H_B(X_g^\dual,X_f)$) is defined as the cohomology of 
 $\CC[(K\oplus K^\dual)_0]\otimes \Lambda^*M_\CC$ by the differential that maps
$$
[m_1\oplus n_1]\otimes P \mapsto
\sum_{m\in \Delta} f(m) [(m+m_1)\oplus n_1] \otimes (m\wedge P)
$$$$+ \sum_{n\in\Delta^\dual} g(n)
[m_1\oplus (n+n_1)]\otimes ({\rm contr.}n)(P)
+[m_1\oplus n_1]\otimes (m_1\wedge P).
$$

\medskip\noindent
{\bf Theorem \ref{mainGKZ}.}
For nondegenerate $f$ and $g$,
the  $B$-space $H_B(X_f,X_g^\dual)$  is naturally isomorphic to
$$
\bigoplus_{ \{0\}\subseteq \theta \subseteq K} R_1(f,\theta)\otimes \widehat {R_1(g,\theta^*)} \otimes \Lambda^{\dim \theta^*} (\CC\theta^*).
$$
The bundle of $H_B(X_f,X_g^\dual)$  over 
the space of nondegenerate $g$ has a natural flat connection.

\medskip\noindent
{\bf Corollary \ref{maincoro5}.}
For strongly nondegenerate $f$ and $g$ the stringy cohomology spaces 
$$H_B(X_f,X_g^\dual)=H_A(X_g^\dual,X_f){\rm~
and~}H_A(X_f,X_g^\dual)=H_B(X_g^\dual,X_f)$$ are equipped with natural
structures of graded associative super-com-mutative algebras.

\bigskip

The flat connection in Theorem \ref{mainGKZ} is in fact explicitly seen to be related to
the GKZ hypergeometric system, or rather its better-behaved version considered in
\cite{BH}. Corollary \ref{maincoro5} is obtained by vertex algebra techniques. It requires 
a technical \emph{strong} non-degeneracy condition developed in \cite{chiralrings}.

We also remark that our notation has $X_f$ and $X_g^\dual$ in it, even though
the varieties $X_f$ and $X_g^\dual$ may not 
exist or be uniquely defined. One might in fact use notation
$H_{A/B}(K,K^\dual,f,g)$ instead. However, we prefer the current notation for the 
following reason. If $\deg\cdot\deg^\dual=1$, then we may interpret 
$X_f$ and $X_g^\dual$ as singular Calabi-Yau hypersurfaces in Gorenstein Fano
toric varieties ${\rm Proj}(\CC[K])$ and ${\rm Proj}(\CC[K^\dual])$ respectively. 
In $\deg\cdot\deg^\dual>1$ case we prefer to keep this notation to stress the relation to 
usual cohomology when it is applicable.

The paper is organized as follows. In Section \ref{inter} we prove a result on intersection 
cohomology of polyhedral fans which is key to the argument and may be of independent interest.
In Section \ref{correction} we prove Theorem \ref{main} as an easy consequence of 
the results of Section \ref{inter}. 
In Section \ref{sec4} we modify our complex so that the new version of stringy 
cohomology comes equipped 
with a flat connection. The discussion culminates in Theorem \ref{mainGKZ}. The first four sections are not particularly technical. Section \ref{voa}  is devoted to the construction of the product structures on the stringy cohomology defined in Section \ref{sec4}. We rely heavily on the results
of \cite{borvert} and \cite{chiralrings}. Finally, in Section \ref{sec6} we outline some open problems related to the construction.

{\bf Acknowledgements.} Joint work with Paul Horja \cite{BH} was one of the key motivations behind this paper. The author thanks Nick Addington for useful comments. The author has been  partially supported by the NSF grants DMS-1003445
and DMS-1201466.

\section{Intersection cohomology of dual cones}\label{inter}
In this section we prove a technical result on intersection cohomology of dual polyhedral 
cones which is the cornerstone of the paper.  We expect the reader to be familiar with the approach 
of \cite{BresslerLunts} and \cite{BBFK} to the intersection cohomology of fans. In these papers the authors 
define first a combinatorial analog of equivariant intersection cohomology of toric varieties as global sections 
of a minimum locally free (in the appropriate sense) and flabby sheaf on the topological space defined by the 
fan. Intersection cohomology is then defined as the quotient of the equivariant intersection 
cohomology by the action of linear functions.

Consider dual strictly convex polyhedral cones $C$ and $C^\dual$ in dual vector spaces $M_\RR$ and $N_\RR$ of dimension 
$r$. We do not assume that $C$ and $C^\dual$ are rational, although this is the case for our primary applications. 
Consider the finite polyhedral fan $\Phi$ in $M_\RR\oplus N_\RR$ 
whose cones are direct sums of faces $\theta\oplus\sigma$ with $\theta\subseteq C, \sigma\subseteq C^\dual$ and 
$\theta\cdot \sigma=0$. Consider the minimum flabby graded locally free sheaf ${\mathcal L}={\mathcal L}_{\{0\}\oplus \{0\}}$
(in the sense of \cite{BresslerLunts}) on $\Phi$. Let $W=\Gamma(\Phi,{\mathcal L})$ be the space of global sections of ${\mathcal L}$.

Note that $W$ is a module over ${\rm Sym}^*(M_\CC\oplus N_\CC)$ which is annihilated by the quadratic equation that gives the 
bilinear pairing. Consider dual bases $(m_i)$ and $(n_i)$ of $M_\RR$ and $N_\RR$
and the corresponding linear functions $y_i$ and $x_i$ on $N_\RR$ and $M_\RR$ respectively. Then $W$ is a module over 
the ring $\CC[x_1,\ldots,x_r,y_1,\ldots, y_r]$ annihilated by $\sum_{i=1}^r x_iy_i$.
We define the differential on $W\otimes \Lambda^*N_\CC$  by 
$$
d=\sum_{i=1}^r x_i\otimes ({\rm contr.} m_i) + \sum_{i=1}^r y_i \otimes (n_i \wedge).
$$
Clearly, $d^2=0$ because $W$ is annihilated by $\sum_{i=1}^r x_iy_i$. In addition, $d$ increases the  degree  by $1$. 
It is also easy to see that $d$ is in fact independent of the choice of dual bases in $M_\RR$ and $N_\RR$.
The main result of this section is the following.
\begin{theorem}\label{key}
Cohomology of $W\otimes \Lambda^*N_\CC$ with respect to $d$ is $0$ for any $C$ in dimension $r>0$.
\end{theorem}

\begin{remark}
For $r=0$, the space $W \otimes \Lambda^*N_\CC$ is one-dimensional with $d=0$.
\end{remark}

\begin{remark}
Theorem \ref{key} is equivalent to the statement that $W$ has finite projective dimension
over the ring $\CC[{\bf x},{\bf y}]/\langle \sum_{i=1}^rx_iy_i\rangle$. We thank Nick Addington for this observation.
It would be interesting to produce a free resolution of $W$ explicitly, but it is not directly related to the goals of this paper.
\end{remark}

\begin{proof}
We can construct ${\mathcal L}={\mathcal L}_{\{0\}\oplus \{0\}}$ as a restriction to $\Phi$
of product of pullbacks of 
${\mathcal L}_{\{0\}}$ sheaves on $K$ and $K^\dual$. As a consequence,
the space $W$ is double graded by degree in $x$ and degree in $y$. The differential $d$ 
preserves the grading by 
\begin{equation}\label{gr}
\deg_x-\deg_y+\deg(\Lambda^*)
\end{equation}
which consequently passes to cohomology. We will show that for $r>0$ this grading on the cohomology is 
at once $< \frac 12r$ and $>\frac 12 r$ to conclude that cohomology is $0$. To accomplish this, we estimate
the cohomology in two different ways by means of spectral sequences associated to two different filtrations of $W$.

Consider the filtration 
$$W=F^{-1}W \supseteq F^0W \supseteq F^1W \supseteq \ldots\supseteq F^rW=0$$
on $W$ defined by 
$$
F^kW = \{w\in W, {\rm~such~that}~w\vert_{\theta\oplus\sigma} = 0~{\rm for~all~}\theta~{\rm of}~\dim\theta \leq k\}
$$
and the corresponding filtration $(F^kW\otimes   \Lambda^*N_\CC)$ on $W\otimes \Lambda^*N_\CC$.
Clearly $(F^kW\otimes   \Lambda^*N_\CC)$ is preserved by $d$ and is compatible with the grading
by $\deg_x-\deg_y+\deg(\Lambda^*)$, in the sense that it induces filtration on each graded component.
We will now consider the spectral sequence associated to this filtration. 

Note that $F^{k}W$ is the kernel of the restriction map from the space of global sections of $\mathcal L$ on $\Phi$
to the space of its sections on the open set $\Phi_k$ (in the fan topology) described by $\dim \theta\leq k$. 
Since $\mathcal L$ is flabby, these restriction maps are surjective. Therefore, the quotient of $F^{k-1}W$ by $F^{k}W$ 
is the kernel of the restriction map $\Gamma(\Phi_{k},{\mathcal L})\to \Gamma(\Phi_{k-1},{\mathcal L})$. 
Note that $\mathcal L$ is built inductively by expanding the open sets. The fan $\Phi_k$ is obtained from the fan
$\Phi_{k-1}$ by attaching cells $\theta\oplus \theta^*$ with $\dim \theta =k$ to the boundary $\partial\theta\oplus\theta^*$.
We thus have
$$
F^{k-1}W/F^{k}W=\bigoplus_{\dim\theta=k} \Gamma(\theta,\partial\theta,{\mathcal L} )
\otimes \Gamma(\theta^*,\mathcal L).
$$
where $\Gamma(\theta,\partial\theta,{\mathcal L} )$ is the kernel of the map
$\Gamma(\theta,{\mathcal L})\to \Gamma(\partial\theta,{\mathcal L})$.
The associated graded complex is then given by 
$$
\bigoplus_\theta  \Gamma(\theta,\partial\theta,{\mathcal L} )
\otimes \Gamma(\theta^*,\mathcal L) \otimes \Lambda^* N_\CC.
$$
For each direct summand we decompose 
\begin{equation}\label{split}
\Lambda^*N_\CC =
\Lambda^* ((\CC\theta)^\dual)\otimes \Lambda^* (\CC\theta^*)
\end{equation}
and see that the $\theta$ component of the associated graded 
complex in question is isomorphic to the tensor
product of two Koszul  complexes. Its cohomology is given by 
$$
\bigoplus_\theta {\rm IH}(\theta,\partial\theta)\otimes {\rm IH}(\theta^*) \otimes \Lambda^{\dim \theta^*} (\CC\theta^*).
$$
where the intersection cohomology spaces above are defined in \cite{BresslerLunts}
as quotients of $\Gamma(\theta,\partial\theta,{\mathcal L} )$
and $\Gamma(\theta^*,\mathcal L)$ by the ideals of linear functions.

We will now use the inequalities on the grading of ${\rm IH}(\theta,\partial\theta)$ and ${\rm IH}(\theta)$ which follow
from the strong Lefschetz theorem \cite{Karu}. Specifically, the graded dimensions of these spaces can be given in
terms of the $G$-polynomials of Stanley  \cite{Stanley}. We see that $\deg_x \geq \frac 12 \dim\theta = \frac 12 k$, 
$\deg_y\leq \frac 12\dim\theta^*=\frac 12(r-k)$. So the degree in the sense of \eqref{gr} is at least
$$
\frac 12 k - \frac 12(r-k) + (r-k) = \frac 12 r.
$$
Note that if $r>0$ then at least one of $\theta$,  $\theta^*$ is of positive dimension, so the corresponding inequality 
is strict. Thus, we get that the degree is $>\frac 12 r$.

By considering the analogous filtration by dimension of $\sigma$, we see that there is a spectral sequence that converges
to the cohomology of $W\otimes \Lambda^*N_\CC$ and starts from
$$
\bigoplus_\theta {\rm IH}(\theta)\otimes {\rm IH}(\theta^*,\partial\theta^*) \otimes \Lambda^{\dim \theta^*} (\CC\theta^*).
$$
The corresponding inequalities $\deg_x\leq \frac 12\dim\theta$ and $\deg_y\geq \frac 12 \dim\theta^*$ (as before, 
at least one of them is sharp) lead to the grading \eqref{gr} being less than $\frac 12r$. This finishes the proof.
\end{proof}

One of the main results of \cite{BresslerLunts} 
is that every locally free flabby sheaf on $\Phi$ is (non-canonically)
isomorphic to a direct sum of copies of minimal locally flabby sheaves that originate at various cones of $\Phi$. Theorem \ref{key} 
is a statement about global sections of such minimal sheaf that originates at the zero cone. In the following proposition we extend 
this result to other minimal flabby locally free sheaves on $\Phi$.
\begin{proposition}\label{maincoro}
Let $C$ and $\Phi$ be as in Theorem \ref{key}.
Fix  $(\theta_0,\sigma_0)$ with $\sigma_0\subseteq \theta_0^*$ and consider the 
minimal flabby locally free sheaf of modules ${\mathcal L}_{(\theta_0\oplus\sigma_0)}$.
If $\sigma_0\subsetneq \theta_0^*$ then the cohomology 
of 
$$
d=\sum_{i=1}^r x_i\otimes ({\rm contr.} m_i) + \sum_{i=1}^r y_i \otimes (n_i \wedge)
$$
on $\Gamma(\Phi, {\mathcal L}_{(\theta_0\oplus\sigma_0)})\otimes \Lambda^*N_\CC$ is zero.
If $\sigma_0= \theta_0^*$ then the cohomology is one-dimensional and can be identified with
$\Lambda^{\dim \theta_0^*} (\CC\theta_0^*)$ as a graded vector space.
\end{proposition}

\begin{proof}
We repeat the same argument as in the proof of Theorem \ref{key}, with the same filtrations 
as before. We have two spectral sequences which converge to the cohomology and start from
$$
\bigoplus_{\theta,\theta_0\subseteq \theta\subseteq \sigma_0^*} 
IH(\theta/\theta_0,\partial\theta/\theta_0)\otimes IH(\theta^*/\sigma_0)\otimes \Lambda^{\dim \theta^*}
(\CC\theta^*)
$$
and
$$
\bigoplus_{\theta,\theta_0\subseteq \theta\subseteq \sigma^*} 
IH(\theta/\theta_0)\otimes IH(\theta^*/\sigma_0,\partial\theta^*/\sigma_0)\otimes \Lambda^{\dim \theta^*}
(\CC\theta^*)
$$
where taking the quotient means considering the image of the appropriate cones modulo the span
of a face.
The grading is estimated to be at least  
$$\frac 12 (\dim \theta-\dim\theta_0) - \frac 12(\dim \theta^*-\dim\sigma_0) + \dim\theta^*
=\frac 12 r +\frac 12(\dim\theta_0+\dim\sigma_0)$$
from one sequence and to be at most the same quantity by the other. If $\sigma_0\subsetneq \theta_0^*$, then the inequalities are 
strict, because at least one of the inclusions $\theta_0\subseteq \theta$ and 
$\sigma_0\subseteq \theta^*$ is proper, which proves the claim.

If $\sigma_0= \theta_0^*$, then only $\theta=\theta_0$ term appears. It remains to observe
that while the tensor product description of \eqref{split} is not canonical, the embedding 
$\CC\theta_0^*\subseteq N_\CC$ is. Thus the one-dimensional space 
$\Lambda^{\dim \theta_0^*} (\CC\theta_0^*)$ is a natural subspace of $\Lambda^*N_CC$. 
Also, the inherent ambiguity
in construction of ${\mathcal L}_{(\theta_0\oplus\sigma_0)}$ does not occur at 
$(\theta_0\oplus\sigma_0)$ itself, where the sections are canonically isomorphic to $\CC$.
\end{proof}

\section{Double Koszul complexes for dual reflexive Gorenstein cones}\label{correction}
In this section we use Theorem \ref{key} and Proposition \ref{maincoro} to prove a result on cohomology 
of a certain complex built from a pair of dual reflexive Gorenstein cones. This fills in
a gap in the proof in \cite{BM}. The reader should be warned however, that this is ultimately
\emph{not} the right complex to consider.  In the next section we will modify the complex somewhat
so that the cohomology  admits a flat connection.

Let $K$ and $K^\dual$ be dual reflexive Gorenstein cones in lattices $M$ and $N$ as in Section \ref{intro},
 $\deg$ and $\deg^\dual$  their degree elements, $\Delta$ and $\Delta^\dual$ the sets of lattice elements of 
degree one and $f$ and $g$  nondegenerate coefficient functions.
As in Section \ref{intro}, we consider the quotient
$\CC[(K\oplus K^\dual)_0]$ of the ring $\CC[K\oplus K^\dual]$ by the ideal spanned by  $[m,n],
m\cdot n >0$. Define 
$$V=\CC[(K\oplus K^\dual)_0]\otimes \Lambda^* N_\CC$$
and an endomorphism $d:V\to V$ by 
\begin{equation}\label{d}
d_{f,g}=\sum_{m\in \Delta} f(m) [m] \otimes ({\rm contr.} m) + \sum_{n\in\Delta^\dual} g(n)[n]\otimes (n\wedge).
\end{equation}
It is easy to see that $d_{f,g}^2=0$. Moreover, $d_{f,g}$ increases by  one the natural grading on $V$ given by 
$[m,n]\otimes P\mapsto m\cdot \deg^\dual +\deg\cdot \,n$.  An explicit description of the  cohomology
of $V$ with respect to differential $d_{f,g}$ 
has been claimed in \cite{BM}, but the proof presented there was incorrect.
Recall that for a face of $\theta\subseteq K$ one can define $R_1(f,\theta)$ as
the image of
$$
\CC[\theta^{\circ}]/ I_{f,\theta} \CC[\theta^{\circ}] \to \CC[\theta]/I_{f,\theta} \CC[\theta]
$$
where $I_{f,\theta}$ is the ideal generated by the logarithmic derivatives of  $\theta$,
see Section \ref{intro}. Similarly, one defines $R_1(g,\theta^*)$ for the dual face 
$\theta^*={\rm Ann}(\theta)\cap K^\dual$.

\begin{theorem}\label{main}
The cohomology of $V$ with respect to $d_{f,g}$ is naturally isomorphic to the direct sum over faces $\theta$ of $K$
$$
\bigoplus_{ \{0\}\subseteq \theta \subseteq K} R_1(f,\theta)\otimes R_1(g,\theta^*) \otimes \Lambda^{\dim \theta^*} (\CC\theta^*).
$$
\end{theorem}

\begin{proof}
Consider the fan of facets of $K\oplus K^\dual$.
Consider the flabby locally free sheaf $\mathcal F$ on it given by $(\theta,\sigma)\to \CC[\theta\oplus\sigma]$.
It can be naturally given a structure of the sheaf of modules over the ring of polynomial functions
via
a ring homomorphism  ${\rm Sym}^*(M_\CC\oplus N_\CC) \to \CC[K\oplus K^\dual]$ defined on generators by
$$(a,b)\mapsto \sum_{m\in\Delta}(m\cdot b) f(m)[m\oplus 0]
+  \sum_{n\in\Delta^\dual}(a\cdot n) g(n)[0\oplus n].
$$
In other words, linear functions on $M$ (resp.  $N$) act as multiplications by the logarithmic partial derivatives
of $f$ (resp. $g$). This induces on  $\mathcal F$ the structure of the module over polynomial functions 
on the faces $\theta\oplus\sigma$ of $K\oplus K^\dual$. In fact, this sheaf is a product of pullbacks of two similarly
defined sheaves ${\mathcal F}_M$ and ${\mathcal F}_N$ on $K$ and $K^\dual$.

The key observation of \cite{BM} is that the non-degeneracy
of $f$ and $g$ implies that  $\mathcal F$ is locally free as a sheaf of modules over the rings of polynomials,
and similarly for ${\mathcal F}_M$ and ${\mathcal F}_N$.
In addition,  ${\mathcal F}_M$ and ${\mathcal F}_N$ are graded by degree. Thus, the machinery of \cite{BresslerLunts} 
applies. We can write (non-canonically) these sheaves as direct sums of copies of graded locally free modules
${\mathcal L}_{\theta}$ and ${\mathcal L}_{\sigma}$ which originate at $\theta$ and $\sigma$ respectively.
We have
\begin{equation}\label{decompose}
{\mathcal F}_M = \bigoplus_{0\subseteq \theta\subseteq K} H_\theta\otimes {\mathcal L}_\theta, \hskip 30pt
{\mathcal F}_N = \bigoplus_{0\subseteq \sigma\subseteq K^\dual} H_\sigma\otimes {\mathcal L}_\sigma
\end{equation}
where $H_\theta$ and $H_\sigma$ are finite-dimensional graded vector spaces.

The space $\CC[(K\oplus K^\dual)_0]$ is the space of global sections of $\mathcal F$ on the open set $\Phi$
that corresponds to $\bigcup_\theta( \theta,\theta^*)$ so by \eqref{decompose}
the space $V=\CC[(K\oplus K^\dual)_0]\otimes \Lambda^* N_\CC$ decomposes as
$$
\bigoplus_{\theta,\sigma} H_\theta\otimes H_\sigma \otimes \Gamma(\Phi, 
{\mathcal L}_\theta \times {\mathcal L}_\sigma)\otimes  \Lambda^* N_\CC
$$
The differential $d_{f,g}$ preserves the direct sum, and in fact acts on each $ \Gamma(\Phi, 
{\mathcal L}_\theta \times {\mathcal L}_\sigma)\otimes  \Lambda^* N_\CC$
as the differential of Proposition \ref{maincoro}. Indeed, for any 
pair of dual bases $(m_i)$ of $M$ and $(n_i)$ of $N$ 
we have 
$$
\sum_{m\in\Delta} f(m)[m]\otimes ({\rm contr.} m) =
\sum_{m\in\Delta} f(m)[m]\otimes \sum_i (m\cdot n_i)({\rm contr.} m_i) 
$$
$$=\sum_i (\sum_{m\in\Delta} f(m) (m\cdot n_i)[m])\otimes ({\rm contr.} m_i) 
=\sum_i x_i\otimes  ({\rm contr.} m_i) 
$$
and similarly for the other part of the differential. 

Let us calculate the cohomology. If
$\sigma\not\subseteq \theta^*$  then  
$\Gamma(\Phi, {\mathcal L}_\theta \times {\mathcal L}_\sigma)$ is $0$,  as is the cohomology. If 
$\sigma \subseteq \theta^*$ then by Proposition \ref{maincoro} the cohomology is only nonzero
for $\sigma=\theta^*$, in which case it can be identified with $\Lambda^{\dim \theta^*} (\CC\theta^*)$.
Thus, as a graded vector space, the cohomology of $V\otimes  \Lambda^* N_\CC$ is
given by 
$$
\bigoplus_\theta H_\theta\otimes H_{\theta^*}\otimes \Lambda^{\dim\theta^*}(\CC\theta^*).
$$
It has been observed in \cite{BM} that $R_1(f,\theta)$ is isomorphic to $H_{\theta}$ as 
as a graded vector space (and similar $R_1(g,\theta^*)$ is isomorphic to $H_{\theta^*}$).
Indeed, in the decomposition \eqref{decompose}, the space $\CC[\theta^\circ]$ is the 
kernel of the restriction map from $\CC[\theta]$ to the boundary, so it is
$H_\theta\otimes {\rm Sym^*}(\CC\theta)^\dual$. We have thus established the isomorphism
of the statement of the theorem, but we have not  showed that it is canonical, 
since the decomposition into direct sum of copies of minimal locally free flabby sheaves is not. 

In order to define this isomorphism canonically, we will use the argument of \cite{BM}. Namely,
there are maps of complexes
\begin{equation}\label{3}
\bigoplus_\theta \CC[\theta^\circ\oplus(\theta^*)^\circ]\otimes \Lambda^*N_\CC 
\to V
\to
\bigoplus_\theta \CC[\theta\oplus\theta^*]\otimes \Lambda^*N_\CC 
\end{equation}
with differentials $d_{f,g}$ defined by the same formula. This induces maps in cohomology,
and the composition of these maps has image isomorphic to 
$$\bigoplus_{ \{0\}\subseteq \theta \subseteq K} R_1(f,\theta)\otimes R_1(g,\theta^*) \otimes \Lambda^{\dim \theta^*} (\CC\theta^*)$$
see \cite{BM}.
Thus, this space is a sub-quotient of the cohomology of $V$, and equality of dimensions
implies that it is isomorphic to the cohomology of $V$, i.e. the first map  of \eqref{3} leads to
a surjective map in cohomology and the second one leads to an injective map in cohomology.
\end{proof} 

\begin{remark}
The arguments of Theorem \ref{main} are applicable to the partial (called deformed in earlier papers) 
lattice algebras $\CC[K\oplus K^\dual]^\Sigma$. In the next 
section that we will see that it  can also 
be used for a slightly modified differential.
\end{remark}

\section{Modified double Koszul complexes and GKZ hypergeometric system}
\label{sec4}

In this section we modify the double Koszul complex of the previous section so that the cohomology
is naturally endowed with a flat connection. The issue at hand is that while the spaces
$R_1(g,\theta)$ form bundles over the space of nondegenerate coefficient functions $g$, there
is no natural connection on them. To create such connection, one needs to modify the spaces slightly by altering the action 
of logarithmic derivatives, as in \cite{BH}.

\begin{definition}
Consider the sheaf of abelian groups $\widehat{\CC[K^\dual]}$
on the fan of faces of $K^\dual$ whose sections
over $\sigma\subseteq K^\dual$ are given by $\CC[\sigma]$. We give it a structure of the 
sheaf of modules over ${\rm Sym}^*((\CC\sigma)^\dual)$ by declaring for each 
$c\in\sigma,~\mu\in( \CC\sigma)^\dual$
$$
\mu \widehat{[c]} =
\sum_{n\in\Delta^\dual\cap \sigma}g(n)\mu(n) \widehat{[n+c]} + \mu(c) \widehat{[c]}.
$$
\end{definition}

\begin{remark}
If not for the term $\mu(c)\widehat{[c]}$, the above is just the structure considered 
in the previous section and \cite{BM}. 
Throughout the rest of the paper, we use  $\widehat{[c]}$ as opposed to 
$[c]$ to signify this new module structure. From now on we will also use the notation
$\widehat{\CC[\sigma]}$ for the space of sections of $\widehat{\CC[K^\dual]}$
on the open subset that corresponds to $\sigma$.
\end{remark}

Note that $\widehat{\CC[K^\dual]}$ is no longer naturally graded. However, it is 
naturally filtered, and the associated graded object is naturally isomorphic to 
${\CC[K^\dual]}$ as a graded module. We can use results about
$\CC[K^\dual]$ to infer statements about  $\widehat{\CC[K^\dual]}$ 
 in view of the following theorem,
proved in \cite{BH}.
\begin{theorem}\cite{BH}\label{BHiso}
There exists a non-canonical isomorphism of sheaves of locally free modules 
$\widehat{\CC[K^\dual]}\simeq \CC[K^\dual]$. Moreover, this isomorphism can be chosen
to act as identity on the associated graded objects. 
\end{theorem}

The following definition mimics the one for $R_1(g,\sigma)$. 
\begin{definition}
For any $\sigma\subseteq K^\dual$ we have a natural inclusion 
$\widehat{\CC[\sigma^\circ]}\to \widehat{\CC[\sigma]}$.
We define $\widehat{R_1(g,\sigma)}$ as the image of 
$$
\widehat{\CC[\sigma^\circ]}/I\widehat{\CC[\sigma^\circ]}\to 
\widehat{\CC[\sigma]}/I\widehat{\CC[\sigma]}
$$
where $I$ is the irrelevant ideal in ${\rm Sym}^*((\CC\sigma)^\dual)$.
\end{definition}

As $g$ varies, the spaces $\widehat{R_1(g,\sigma)}$ can be given a natural flat connection 
based on the following result of \cite{BH}. 
\begin{proposition}\cite{BH}\label{flat}
The space $\widehat{R_1(g,K^\dual)}$ is naturally isomorphic to the dual of the space of 
solutions to the better-behaved GKZ hypergeometric system  ${\rm bbGKZ}(\Delta^\dual,
(K^\dual)^{\circ},\beta=0)$ that can be extended to the solutions of
${\rm bbGKZ}(\Delta^\dual,
K^\dual,\beta=0)$, in the neighborhood of $g$. The space has a natural filtration
by the order of vanishing at $g$ and the associated graded space is naturally isomorphic 
to $R_1(g,K^\dual)$.
\end{proposition}

Indeed, we can define the connection by declaring the solutions to the aforementioned
equations to be flat. For $\sigma\subsetneq K^\dual$ we can can simply consider 
the solutions to be independent of the parameter $g(v)$ for $v\not\in\sigma$.
We remark that these spaces have a geometric interpretation in terms of cohomology
of hypersurfaces in tori with the Gauss-Manin connection,
see \cite{Batduke} but we will not focus on this.

We are now ready to define and calculate the main objects of interest, namely the 
stringy cohomology spaces associated to dual Gorenstein cones.
\begin{definition}\label{defmain}
Let $K$ and $K^\dual$ be dual reflexive Gorenstein cones. Define as before the space 
$V=\CC[(K\oplus K^\dual)_0]\otimes \Lambda^* N_\CC$. Consider the 
differential $\widehat d_{f,g}$ on it given by 
$$
\widehat d_{f,g} ([m_1\oplus n_1]\otimes P) = d_{f,g}([m_1\oplus n_1]\otimes P)+ [m_1\oplus n_1]\otimes (n_1\wedge P)
$$
$$
=\sum_{m\in \Delta} f(m) [(m+m_1)\oplus n_1] \otimes ({\rm contr.} m)(P) \hskip 100pt
$$$$+ \sum_{n\in\Delta^\dual} g(n)
[m_1\oplus (n+n_1)]\otimes (n\wedge P)
+[m_1\oplus n_1]\otimes (n_1\wedge P)
.
$$
Then we define stringy cohomology $B$-space $H_{B}(X_f,X_g^\dual)$ as the cohomology
of $V$ with respect to $\widehat d_{f,g}$, with the grading given by 
$$
[m\oplus n]\otimes P \mapsto 2 \,m\cdot \deg^\dual + \deg(P).
$$
We will also call this the $A$-space of the mirror pair $(X_g^\dual,X_f)$. Similarly, 
 $H_A(X_f,X_g^\dual)$ (equal to  $H_B(X_g^\dual,X_f)$) is defined as the cohomology of 
 $\CC[(K\oplus K^\dual)_0]\otimes \Lambda^*M_\CC$ by the differential that maps
$$
[m_1\oplus n_1]\otimes P \mapsto
\sum_{m\in \Delta} f(m) [(m+m_1)\oplus n_1] \otimes (m\wedge P)
$$$$+ \sum_{n\in\Delta^\dual} g(n)
[m_1\oplus (n+n_1)]\otimes ({\rm contr.}n)(P)
+[m_1\oplus n_1]\otimes (m_1\wedge P).
$$
\end{definition}

It is straightforward to see that $\widehat d_{f,g}$ is a differential. Moreover, we can describe
its cohomology explicitly. 

\begin{theorem}\label{mainGKZ}
For nondegenerate $f$ and $g$,
the  $B$-space $H_B(X_f,X_g^\dual)$  is naturally isomorphic to
$$
\bigoplus_{ \{0\}\subseteq \theta \subseteq K} R_1(f,\theta)\otimes \widehat {R_1(g,\theta^*)} \otimes \Lambda^{\dim \theta^*} (\CC\theta^*).
$$
The bundle of $H_B(X_f,X_g^\dual)$  over 
the space of nondegenerate $g$ has a natural flat connection.
\end{theorem}

\begin{proof}
We have the maps of complexes
$$
\bigoplus_\theta \CC[\theta^\circ]\otimes \widehat{\CC[(\theta^*)^\circ]}\otimes \Lambda^*N_\CC 
\to V
\to
\bigoplus_\theta \CC[\theta]\otimes \widehat{\CC[\theta^*]}\otimes \Lambda^*N_\CC 
$$
similar to \eqref{3}. Then, as in the proof of Theorem \ref{main} the space
$\bigoplus_{ \{0\}\subseteq \theta \subseteq K} R_1(f,\theta)\otimes \widehat {R_1(g,\theta^*)} 
\otimes \Lambda^{\dim \theta^*} (\CC\theta^*)$ is 
a sub-quotient of the cohomology of $V$ with respect to $\widehat d_{f,g}$. 
In view of Theorem \ref{BHiso}, there is an isomorphism between \eqref{3} and the 
above maps of complexes. This implies that the left map induces a surjection in
the cohomology and the right map induces an injection. This proves the first assertion of the 
theorem, and the flat connection statement follows from Proposition \ref{flat}.
\end{proof}

\begin{remark}
The stringy cohomology spaces  $H_{A/B}(X_f,X_g^\dual)$ are only single-graded
rather than double-graded. However, they possess natural filtrations such that the 
associated graded spaces are those described in \cite{BM}. Thus,  spaces 
$H_{A/B}(X_f,X_g^\dual)$ are stringy analogs of de Rham cohomology
while the spaces of \cite{BM} are stringy analogs of Dolbeault cohomology.
\end{remark}

\section{Product structure on the stringy cohomology spaces} \label{voa}
The previous candidates for the stringy cohomology spaces that were considered
in \cite{BM} have been given a structure of super-commutative algebra (in fact, 
two structures depending on whether $A$ or $B$ version are considered) by 
identifying them with chiral rings of a certain $N=2$ vertex algebra.
We would like to extend these results of \cite{borvert} and \cite{chiralrings}
to the newly constructed spaces $H_B(X_f,X_g^\dual)$ and $H_A(X_f,X_g^\dual)$.
What we find in the process is that contrary to the situation we have encountered
so far in this paper, the $~\widehat{}~$ analog of the vertex algebra of mirror symmetry 
is drastically different from the one previously considered. In fact, it is just the stringy
cohomology space itself.\footnote{This should be viewed as analogous to the statement
of \cite{MSV} that chiral de Rham complex is quasi-isomorphic to the usual de Rham
complex.} Working knowledge of vertex algebras is assumed for this more technical 
section.

Let us first recall the construction of \cite{borvert} and \cite{chiralrings}, focusing
on the $B$ space. Let $M$, $N$, $K$, $K^\dual$, $\Delta$, $\Delta^\dual$, $f$ and $g$
be as before. Consider first the lattice vertex algebra ${\rm Fock}_{M\oplus N}$.
It is generated by the fields 
$$
m^{bos}(z),~n^{bos}(z),~ \ee^{\int m^{bos}(z)},~ \ee^{\int n^{bos}(z)},~
m^{ferm}(z),~n^{ferm}(z)
$$
in the world-sheet variable $z$. 

Consider the differential $D_{f,g}$ on ${\rm Fock}_{M\oplus N}$ 
$$
D_{f,g}:={\rm Res}_{z=0}\Big(
\sum_{m\in \Delta} f(m)m^{ferm}(z)\ee^{\int m^{bos}(z)} 
+
\sum_{n\in \Delta^\dual} g(n)n^{ferm}(z)\ee^{\int n^{bos}(z)} 
\Big).
$$
We denote by $V_{f,g}$ the cohomology of ${\rm Fock}_{M\oplus N}$
with respect to $D_{f,g}$.
It carries a natural structure of $N=2$ vertex algebra, induced from
${\rm Fock}_{M\oplus N}$ by
$$
\begin{array}{rcl}
G^+(z)&=&\sum_k (n^k)^{bos}(z)(m^k)^{ferm}(z)-\partial_z\deg^{ferm}(z)
\\
G^-(z)&=&\sum_k (m^k)^{bos}(z)(n^k)^{ferm}(z)-\partial_z(\deg^\dual)^{ferm}(z)
\\
J(z)&=&\sum_k (m^k)^{ferm}(z)(n^k)^{ferm}(z)+\deg^{bos}(z)-
(\deg^\dual)^{bos}(z)
\\
L(z)&=&\sum_{k}(m^k)^{bos}(z)(n^k)^{bos}(z)
+\frac12\sum_k\partial_z(m^k)^{ferm}(z)(n^k)^{ferm}(z)\\
&&
\hskip -50pt
-\frac12\sum_k
(m^k)^{ferm}(z)\partial_z(n^k)^{ferm}(z)
-\frac12\partial_z \deg^{bos}(z)-\frac12\partial_z(\deg^\dual)^{bos}(z)
\end{array}
$$
The operators $L[0]={\rm Res}_{z=0}(zL(z))$ and $J[0]={\rm Res}_{z=0}J(z)$ play 
a special role and are called conformal weight and fermion numbers respectively.
The main result of \cite{chiralrings} is the following.
\begin{theorem}\cite{chiralrings}
For strongly nondegenerate $f$ and $g$, the above $N=2$ structure on $V_{f,g}$
is of sigma model type. This means that the gradings by 
 $H_A=L[0]-\frac 12 J[0]$ and 
$H_B=L[0]+\frac 12 J[0]$ are integer and nonnegative.
\end{theorem}

\begin{remark}
We refer the reader to \cite{chiralrings} for the technical definition of 
\emph{strong} non-degeneracy. We will not need the precise statement here.
\end{remark}

Every vertex algebra of sigma model type has two natural subspaces, called 
$A$ and $B$ chiral rings characterized as kernels of $H_A$ and $H_B$ respectively.
In the case of $V_{f,g}$ these have been computed to coincide with 
the stringy cohomology spaces suggested in \cite{BM}. The computation is based 
on the following result.
\begin{theorem}\label{mainchiralrings}\cite{chiralrings}
Let $f$ and $g$ be strongly non-degenerate.
Then $D_{f,g}$-cohomology of ${\rm Fock}_{M\oplus N}$ 
has only nonnegative integer eigenvalues of $H_A$.
Moreover, the $H_A=0$ eigenspace comes from
${\rm Fock}_{K\oplus (K^\dual-\deg^\dual)}$. The operator $H_B$
also has only nonnegative integer eigenvalues on $V_{f,g}$
and its kernel comes from ${\rm Fock}_{(K-\deg)\oplus K^\dual}$.
\end{theorem}

As a corollary, the kernel of $H_B$ is isomorphic to the cohomology
of $D_{f,g}$ of the $H_B=0$ subspace of ${\rm Fock}_{(K-\deg)\oplus K^\dual}$.
This subspace is seen to be spanned by elements that correspond to fields
of the form\footnote{with the normal ordering implicitly assumed}
\begin{equation}\label{B}
n_1^{ferm}(z)\cdots n_k^{ferm}(z)\ee^{\int (m-\deg)^{bos}(z)+n^{bos}(z)}
\end{equation}
for $n_i\in N_\CC$, $m\in K$, $n\in K^\dual$, $m\cdot n=0$. The space of these
fields is naturally isomorphic to $\CC[(K\oplus K^\dual)_0]\otimes \Lambda^*N_\CC$,
with the action of $D_{f,g}$ given by $d_{f,g}$ from \eqref{d}. This leads to the 
following result.
\begin{corollary}\cite{chiralrings}\label{Balg}
For strongly nondegenerate $f$ and $g$, the cohomology of $\CC[(K\oplus K^\dual)_0]\otimes \Lambda^*N_\CC$ with respect to $d_{f,g}$ is naturally identified with the $B$ chiral ring
of $V_{f,g}$. It thus carries a natural structure of graded associative algebra.
\end{corollary}

\begin{remark}
It is currently very unclear how to write the product implied by the above corollary without
the full force of vertex algebra machinery. The trouble is that as one takes the operator
product expansions of the fields of \eqref{B}, one typically lands in 
${\rm Fock}_{(K-2\deg)\oplus K^\dual}$ rather than 
${\rm Fock}_{(K-\deg)\oplus K^\dual}$. It then takes delicate vertex algebra considerations to
prove that these fields can be reduced to fields from ${\rm Fock}_{(K-\deg)\oplus K^\dual}$
in $D_{f,g}$ cohomology.
\end{remark}

We would like to extend the results of Theorem \ref{mainchiralrings} and Corollary \ref{Balg}
to the new version of the differential $\widehat d_{f,g}$ on 
$\CC[(K\oplus K^\dual)_0]\otimes \Lambda^*N_\CC$. The first observation is that the 
additional term in it is actually a residue of a very natural field. For an element 
$[m\oplus n]\otimes P$ of $\CC[(K\oplus K^\dual)_0]\otimes \Lambda^*N_\CC$
we denote the corresponding field by $P(N^{ferm})(z) \ee^{\int (m-\deg)^{bos}(z)+n^{bos}(z)}$.
\begin{proposition}\label{routine}
The action of ${\rm Res}_{z=0} G^-(z)$ on a field
$$
P(N^{ferm})(z) \ee^{\int (m-\deg)^{bos}(z)+n^{bos}(z)} 
$$
is given by 
$$
(n\wedge P)(N^{ferm})(z) \ee^{\int (m-\deg)^{bos}(z)+n^{bos}(z)}.
$$
\end{proposition}

\begin{proof}
This is a routine lattice vertex algebra calculation, which we nonetheless 
include for the benefit of the reader.
The $\partial_z(\deg^\dual)^{ferm}(z)$ in $G^-(z)$ does not contribute to the residue at
zero and can be ignored. Then we have the operator product expansion
$$
\sum_k (m^k)^{bos}(z)(n^k)^{ferm}(z) 
P(N^{ferm})(w) \ee^{\int (m-\deg)^{bos}(w)+n^{bos}(w)} 
$$
$$\sim (z-w)^{-1}\sum_k (m^k\cdot n)(n^k)^{ferm}(w) 
P(N^{ferm})(w) \ee^{\int (m-\deg)^{bos}(w)+n^{bos}(w)} 
$$
$$
=(z-w)^{-1}n^{ferm}(w)P(N^{ferm})(w) \ee^{\int (m-\deg)^{bos}(w)+n^{bos}(w)}
$$
$$
=(z-w)^{-1}(n\wedge P)(N^{ferm})(z) \ee^{\int (m-\deg)^{bos}(z)+n^{bos}(z)},
$$
which verifies the statement of the proposition.
\end{proof}

\begin{remark}
The operator  ${\rm Res}_{z=0}G^-(z)$ has an important geometric meaning. In
terms of $X_g^\dual$ it roughly corresponds to the de Rham differential.
Also note that cohomology of $V_{f,g}$ with respect to ${\rm Res}_{z=0}G^-(z)$
is the $B$ chiral ring of $V_{f,g}$.
\end{remark}

It is therefore reasonable to add the residue of $G^-(z)$ to the definition 
of $D_{f,g}$.
\begin{definition}\label{defwidehatv}
We denote by $\widehat V_{f,g;B}$ the cohomology of 
${\rm Fock}_{M\oplus N}$ by the differential $\widehat D_{f,g;B}$ given by
$$
{\rm Res}_{z=0}\Big(
\sum_{m\in \Delta} f(m)m^{ferm}(z)\ee^{\int m^{bos}(z)} 
+ \sum_{n\in \Delta^\dual} g(n)n^{ferm}(z)\ee^{\int n^{bos}(z)} 
$$
$$+\sum_k (m^k)^{bos}(z)(n^k)^{ferm}(z)-\partial_z(\deg^\dual)^{ferm}(z)
\Big)
$$
for dual bases $(m^k)$ and $(n^k)$ of $M$ and $N$. 
\end{definition}

We will now show that $\widehat V_{f,g;B}$
is naturally isomorphic to the 
stringy cohomology $B$-space $H_B(X_f,X_g^\dual)$.
\begin{theorem}\label{mainthm}
For strongly nondegenerate $f$ and $g$ the cohomology $\widehat V_{f,g;B}$  
of Definition \ref{defwidehatv} is naturally isomorphic to $H_B(X_f,X_g^\dual)$.
\end{theorem}

\begin{proof}
The proof proceeds in several steps. First, we show that the cohomology of 
${\rm Fock}_{M\oplus N}$ with respect to $\widehat D_{f,g;B}$ is naturally 
isomorphic to the cohomology of ${\rm Fock}_{M\oplus K^\dual}$ with 
respect to the same operator. The analogous 
fact for the cohomology with respect to $D_{f,g}$ is  proved in \cite[Proposition 8.2]{borvert}.

The idea of \cite[Proposition 8.2]{borvert} is to construct operators $R$ such
that the anticommutator of $R$ with $\widehat D_{f,g;B}$ is equal to identity plus
an operator that increases a grading with respect to a ray of $K$. It remains to
observe that the operators $R$ constructed in \cite[Proposition 8.1]{borvert} anti-commute with 
$G^-(z)$, because the former are made from $\ee^{\int M^{bos}(z)}$ and $N^{ferm}(z)$ 
fields only, and the latter is made from $M^{bos}(z)$ and $N^{ferm}(z)$. Thus the
same fields $R$ can be used in the new argument.

Now we are dealing with cohomology of ${\rm Fock}_{M\oplus K^\dual}$ with 
respect to $\widehat D_{f,g;B}$. The method of \cite{chiralrings} was to employ
a spectral sequence for the double complex, with the $f$ and $g$ parts of the 
differential acting as horizontal and vertical arrows. This spectral sequence 
does not converge on the whole ${\rm Fock}_{M\oplus N}$, but does converge
on ${\rm Fock}_{M\oplus K^\dual}$.
The situation is a bit different in the current setting, because the additional term
in $\widehat D_{f,g;B}$ destroys the double grading by degree in $M$ and $N$,
however there is still a spectral sequence which we will describe below if we 
use a more subtle double grading.

Consider the operator $H_A$ on ${\rm Fock}_{M\oplus K^\dual}$ that comes
from the $N=2$ structure. Note that  $D_{f,g}$ preserves eigenvalues of $H_A$ 
and ${\rm Res}_{z=0}G^-(z)$ increases $H_A$ by one, as part of the commutator relations
of the $N=2$ algebra. Thus, we can think of $\widehat D_{f,g;B}$ as a differential
of the total complex of the double complex with the double grading
$(H_A+\bullet\cdot \deg^\dual,\deg\cdot\bullet)$ and 
$$
d_1={\rm Res}_{z=0}\Big(
\sum_{m\in \Delta} f(m)m^{ferm}(z)\ee^{\int m^{bos}(z)} 
+G^-(z)
\Big)
$$$$
d_2 = {\rm Res}_{z=0} \sum_{n\in \Delta^\dual} g(n)n^{ferm}(z)\ee^{\int n^{bos}(z)}.
$$
The corresponding spectral sequence that starts with 
$
H_{d_2}({\rm Fock}_{M\oplus K^\dual}) 
$
then converges, since the second grading is nonnegative on 
${\rm Fock}_{M\oplus K^\dual}$. The cohomology with respect to $d_2$  has been
studied in \cite{chiralrings}. In particular, the $H_B=0$ part of the cohomology is zero
for ${\rm Fock}_{c\oplus K^\dual}$ for $c\not\in K-\deg$, provided $g$ is strongly nondegenerate.
Then, as in \cite{chiralrings} we conclude that the $H_B=0$ part of the cohomology
of ${\rm Fock}_{M\oplus K^\dual})$ with respect to $\widehat D_{f,g;B}$
is the $H_B=0$ part of the cohomology of ${\rm Fock}_{K-\deg\oplus K^\dual})$.
As a corollary of Proposition \ref{routine}, we see that the $H_B=0$ part 
of the cohomology of ${\rm Fock}_{M\oplus N}$ with respect to $\widehat D_{f,g;B}$
is naturally isomorphic to $H_B(X_f,X_g^\dual)$. 

It remains to show that $H_B\neq 0$ part of the cohomology vanishes,
which is only a feature of $\widehat V_{f,g;B}$ rather than $V_{f,g}$.
We  use ${\rm Res}_{z=0}zG^{+}(z)$ which anti-commutes with 
the $f$ and $g$ terms of $\widehat D_{f,g;B}$. As part of the 
$N=2$ algebra structure, the anticommutator of ${\rm Res}_{z=0}zG^{+}(z)$ with
the ${\rm Res}_{z=0}G^{-}(z)$  is $H_B$. This shows that the 
cohomology occurs only at $H_B=0$, since in other eigenspaces of $H_B$
the operator ${\rm Res}_{z=0}zG^{+}(z)$ provides homotopy to the identity.
\end{proof}

\begin{corollary}\label{maincoro5}
For strongly nondegenerate $f$ and $g$ the stringy cohomology spaces 
$$H_B(X_f,X_g^\dual)=H_A(X_g^\dual,X_f){\rm~
and~}H_A(X_f,X_g^\dual)=H_B(X_g^\dual,X_f)$$ are equipped with natural
structures of graded associative super-com-mutative algebras.
\end{corollary}

\begin{proof}
It is sufficient to consider $H_B(X_f,X_g^\dual)= \widehat V_{f,g;B}$, the other is obtained by
switching $M$ and $N$. Take two 
$\widehat u$ and $\widehat v$ elements in $\widehat V_{f,g;B}$, and 
lift them to $u,v\in {\rm Fock}_{M\oplus N}$. Then take the operator product expansion
of the corresponding fields 
$$
u(z)v(w) = \sum_{k\geq -l}(z-w)^k c_k(w).
$$
Each $c_k$ is annihilated by $\widehat D_{f,g;B}$ and thus descends to cohomology.
Moreover, only $c_0$ has $H_B=0$ and can be nonzero in the cohomology.
We define the product by $\widehat u*\widehat v:=\widehat c_0$. Standard vertex 
algebra considerations then assure that this product is (super)commutative and associative.
\end{proof}

\begin{remark}
The $H_B$ ring is in fact graded by what was called in \cite{chiralrings} the sum
of fermion number and cohomology grading. It coincides with the total grading
of the double complex in the proof of Theorem \ref{mainthm}. It also coincides 
with the grading of Definition \ref{defmain}. 
\end{remark}

\section{Concluding remarks and open questions}\label{sec6}
One can glean from the argument that in order 
to define the product structure on $H_B(X_f,X_g^\dual)$
one does not need both $f$ and $g$ to be strongly nondegenerate.
Rather, one needs $g$ to be strongly nondegenerate
while usual non-degeneracy suffices for $f$.
The spaces themselves are  defined and are of constant dimension 
with just a simple non-degeneracy assumption on both $f$ and $g$.
It is possible that strong non-degeneracy is an artifact of the argument, rather than a consequence
of product structure acquiring poles at nondegenerate by not strongly nondegenerate $g$.

Can one describe the product on the stringy cohomology in terms of commutative algebra,
thus avoiding vertex algebra machinery? This has been a recurring question for the last ten years
or so. However, it is possible that the extra term ${\rm Res}_{z=0} G^-(z)$ in the differential may 
allow for an easier reduction from ${\rm Fock}_{(K-2\deg)\oplus K^\dual}$ to 
${\rm Fock}_{(K-\deg)\oplus K^\dual}$ which lies at the heart of the matter.

The structures described in this paper correspond to small quantum cohomology.
What is the big quantum cohomology for these examples? In other words, we need to 
construct Frobenius manifolds based on the stringy cohomology.

Similarly one can ask whether the construction of this paper can lead to an algebraic 
construction of the $A_\infty$ categories of boundary conditions in the open string theory.
When dual Gorenstein cones $(K,K^\dual)$  lead to complete intersection Calabi-Yau varieties,
this would amount to an alternative description of Fukaya category and derived category 
of coherent sheaves, and might lead to a verification of Homological Mirror Symmetry.

The dual Gorenstein cones construction has been recently generalized in two directions. 
In \cite{BorBH}, the duality condition has been relaxed a bit, to include the examples of Berglund
and H\"ubsch. It appears that the arguments of this paper should extend to
the setting of almost dual Gorenstein cones, but this needs to be checked. Another generalization, studied in \cite{BK}, modifies the differential $D_{f,g}$ by allowing more general linear combinations of the fermionic fields. The resulting algebra no longer has $N=2$ structure 
 and is related to $N=(0,2)$ nonlinear sigma models.  It would be interesting to 
see to what extent the definitions of this paper extend to this more general setting.

\end{document}